\documentclass[11pt]{amsart}
\usepackage{amssymb}
\usepackage{amsmath}
\usepackage{amsrefs}
\usepackage{blindtext}
\usepackage{multicol}
\usepackage{mathrsfs}
\usepackage{adjustbox}
\usepackage{tikz}
\usepackage{graphics}
\usepackage{mathtools}
\usepackage{thmtools}
\usepackage{mathrsfs,amsmath, tikz,verbatim, wasysym,xcolor, amsfonts, amssymb, amsthm, graphicx, hyperref, enumerate, xfrac, xspace, mathtools}
\usepackage{thm-restate}
\usetikzlibrary{arrows}
\usetikzlibrary{shapes.geometric}
\usetikzlibrary{matrix}
\usetikzlibrary{decorations.pathreplacing, decorations.markings}
\usepackage[utf8]{inputenc}
\usepackage[top=1in, bottom=1in, right=1in, left=1in]{geometry}
\usepackage{hyperref}
\usepackage{graphicx}
\usepackage{xcolor}
\usepackage{tikz-cd}

\usepackage[capitalise]{cleveref}

\newtheorem{theorem}{Theorem}[section]
\newtheorem{lem}[theorem]{Lemma}
\newtheorem{proposition}[theorem]{Proposition}
\newtheorem{cor}[theorem]{Corollary}

\newtheorem{question}[theorem]{Question}

\theoremstyle{definition}
\newtheorem{dfn}[theorem]{Definition}
\newtheorem{ex}[theorem]{Example}

\newtheorem{ntn}[theorem]{Notation}

\numberwithin{theorem}{section}
\allowdisplaybreaks

\newenvironment{theorem_no_number}[1][]{\begin{trivlist}
\item[\hskip \labelsep {\bfseries Theorem \def\temp{#1}\ifx\temp\empty  #1\else  #1\fi
.}] \itshape}  {\end{trivlist}}
\newenvironment{cor_no_number}[1][]{\begin{trivlist}
\item[\hskip \labelsep {\bfseries Corollary \def\temp{#1}\ifx\temp\empty  #1\else  #1\fi
.}] \itshape}  {\end{trivlist}}
\numberwithin{theorem}{section}

\DeclareMathOperator{\supt}{supt}
\DeclareMathOperator{\expsum}{expsum}

\DeclareMathOperator{\Mon}{Mon}

\newcommand{\Z}{\mathbb{Z}}
\newcommand{\N}{\mathbb{Z}_{\geq 0}}

\newenvironment{AL}{\noindent\color{red} AL:}{}

\newcommand{\makeset}[2]{\left\lbrace #1 \;\middle|\;
  \begin{tabular}{@{}l@{}}
    #2
   \end{tabular}
  \right\rbrace}

\title{The Diophantine problem in Thompson's group \(F\)}
\author{Luna Elliott and Alex Levine}


\begin{document}

\begin{abstract}
    We show that the Diophantine problem in Thompson's group \(F\) is undecidable. Our
    proof uses the facts that \(F\) has finite commutator width and rank \(2\)
    abelianisation, then uses similar arguments used by Büchi and Senger and Ciobanu
    and Garreta to show the Diophantine problem in free groups and monoids with abelianisation
    constraints is undecidable. 
\end{abstract}

\maketitle
\section{Introduction}
The \emph{Diophantine problem} in a group \(G\) is the decision question which takes as input a
finite system \(\mathcal E\) of equations in \(G\) and outputs whether or not \(\mathcal E\) admits a solution. One of the first major breakthroughs was made by Makanin, when he proved that the Diophantine problem in free groups was decidable \cite{Makanin_eqns_free_group, Makanin_semigroups, Makanin_systems}. Decidability for torsion-free hyperbolic groups followed in \cite{RipsSela}, and generic hyperbolic groups in \cite{DahmaniGuirardel}. In addition to these, the decidability of the Diophantine problem has been shown for right-angled Artin groups \cite{DiekertMuscholl}, virtually direct products of hyperbolic groups \cite{CiobanuHoltRees} and free products of free and finite groups
\cite{DiekertLohrey}, amongst others. On the other hand, Roman'kov gave an example of a nilpotent group with undecidable
Diophantine problem \cite{Romankov}, with Duchin, Liang and Shapiro proving that all free nilpotent groups of rank and class at least \(2\) have an undecidable Diophantine problem \cite{DuchinLiangShapiro}. `Many' random nilpotent groups have an undecidable Diophantine problem \cite{GarretaMiasnikovOvchinnikov17}. Non-abelian free metabelian groups have undecidable Diophantine problem \cite{Romankov79}, as do many groups defined by metabelian presentations \cite{GarretaMiasnikovOvchinnikov21}. Recently, Dong has shown that the Diophantine problem in the restricted wreath product \(\mathbb{Z} \wr \mathbb{Z}\) is undecidable \cite{Dong}.

The purpose of this paper is to show that the Diophantine problem is undecidable in Thompson's group \(F\). Thompson's group \(F\) is a finitely presented group of piecewise linear homomorphisms of
the unit interval, which has received significant attention since its introduction by Richard
Thompson \cite{Thompson}. Whilst \(F\) is not simple, it was introduced alongside Thompson's groups
\(T\) and \(V\), which were the first examples of infinite, finitely presented simple groups, and
the three groups exhibit certain common behaviour. Thompson's group \(F\) has a decidable conjugacy
problem (since it is a diagram group \cite{GubaSapir97}, see also \cite{BelkMatucci}), and so this article adds it to the known cases of groups with decidable conjugacy problem,
but undecidable Diophantine problem. Our proof makes use of the fact that the abelianisation of
\(F\) is \(\mathbb{Z}^2\) \cite{CannonFloydParry}, and that every element of the commutator subgroup of \(F\) is a product of at most \(2\) commutators \cite{commwidth}. Further properties of \(F\) include that
the commutator subgroup is simple, the growth rate is exponential and that it does not contain
a non-abelian free subgroup \cite{Thompson, CannonFloydParry}. The last, together with the fact that \(F\) is not
elementary amenable, has led to much interest in whether \(F\) is either amenable or non-amenable. Thompson's group \(F\) also has a deep relationship with the concept of associativity \cite{Dehornoy2005, Dehornoy96, McKenzieThompson, Brin05}.

 \begin{theorem}\label{main-thm}
        The Diophantine problem in Thompson's group \(F\) is undecidable.
\end{theorem}

James Belk and Corentin Bodart both independently pointed out to us that Thompson's group \(F\) can be realised as a centraliser in Thompson's group \(T\), giving the undecidability of the
Diophantine problem in Thompson's group \(T\). Some arguments in a similar direction were used
by Lasserre when discussing the first-order theories of \(F\) and \(T\) \cite{Lasserre}, although these do not transfer to Diophantine problems. 

\begin{cor}\label{main-cor}
    The Diophantine problem in Thompson's group \(T\) is undecidable.
\end{cor}

Whilst the conjugacy problem is a special case of the Diophantine problem, as it can be decided using
the equation \(X^{-1}gXh^{-1} = 1\), one can also generalise equations to generic first-order sentences. The weaker statement that the first-order theory of \(F\) is undecidable was first proved by Bardakov and Tolstykh \cite{BardakovTolstykh}. Their method of proving that \(F\) has an undecidable first-order theory involved
reducing the question to the first-order theory of the restricted wreath product \(\Z \wr \Z\). Given
that \(\Z \wr \Z\) has recently been shown to have an undecidable Diophantine problem \cite{Dong}, one might
hope that this proof can be adapted to reduce the Diophantine problem of  \(F\) to that of \(\Z \wr \Z\). The proof requires the fact that centralisers of
first-order definable sets are themselves first-order definable. Centralisers of some subsets Badakov and Tolstykh use can be defined using equations as well, however, it is not clear if all
such sets can be, and so we use an alternative approach.

The Diophantine problem with abelianisation constraints in non-commutative free monoids was shown to be
undecidable by Büchi and Senger \cite{BuchiSenger}. Using the fact that every element of the commutator
subgroup of Thompson's group \(F\) is a product of at most two commutators \cite{commwidth}, we can simulate abelianisation constraints in \(F\) using
systems of equations. We can then employ a similar argument to \cite{BuchiSenger, CiobanuGarreta}, which uses the fact that
the abelianisation of \(F\) is \(\Z^2\), to show that
the Diophantine problem in \(F\) is undecidable. The argument involves reducing Hilbert's tenth
problem over the non-negative integers to the Diophantine problem of \(F\), and thus
Matijasevic's famous negative answer to Hilbert's tenth problem \cite{Matijasevic70} completes the proof. Naturally, there are many classes of groups where the Diophantine problem with
abelianisation constraints is decidable, such as virtually abelian groups \cite{CiobanuEvettsLevine}, or hyperbolic groups
with finite abelianisation \cite{CiobanuGarreta}, so this approach will not work for all infinite groups with finite commutator width. In particular, since Thompson's group \(V\) has a trivial
abelianisation, there is no reason to expect any argument in this direction for \(V\)
will yield results.

\section{Preliminaries}
 We will use right actions throughout. That is, we write \((x)f\) as opposed to \(f(x)\) when applying a function \(f\) to a point \(x\), and compose functions from left to right.
\subsection{Thompson's groups \(F\) and \(T\)}
    We give a brief introduction to Thompson's groups \(F\) and \(T\). We refer the reader to \cite{CannonFloydParry} for more details.

    \begin{dfn}\label{FDef}
        \emph{Thompson's group \(F\)} is the group of piecewise linear homeomorphisms from
        \([0, 1]\) to itself, differentiable everywhere except at finitely many dyadic rational
        numbers, and such that the derivative on every differentiable subinterval is always
        an integer power of \(2\).

        If \(f \in F\) the \emph{support} of \(f\) is the set \(\supt(f) = \{x \in [0, 1] \mid (x)f \neq x\}\).
    \end{dfn}

    \begin{ex}
        Consider the elements \(x_0\) and \(x_1\) of Thompson's group \(F\), defined by
        \[
        (p)x_0 = \left\{
                \begin{array}{cl}
                   2p  & p \in [0, \frac{1}{4}] \\
                   p+\frac{1}{4}  & p \in [\frac{1}{4}, \frac{1}{2}] \\
                  \frac{p}{2}+\frac{1}{2} & p \in [\frac{1}{2}, 1],
                \end{array}
            \right.
            \qquad
            (p)x_1 = \left\{
                \begin{array}{cl}
                  p & p \in [0, \frac{1}{2}] \\
                  2p-\frac{1}{2}  & p \in [\frac{1}{2}, \frac{5}{8}] \\
                   p+\frac{1}{8}  & p \in [\frac{5}{8}, \frac{3}{4}] \\
                  \frac{p}{2}+\frac{1}{2} & p \in [\frac{3}{4}, 1],
                \end{array}
            \right.
        \]
        In these cases, \(\supt(x_0) = (0, 1)\) and \(\supt(x_1) = (\frac{1}{2}, 1)\).
    \end{ex}

    \begin{theorem}[{\cite[Corollary 2.6 and proof of Theorem 4.1]{CannonFloydParry}}]\label{x0andx1gen}
        \label{F:abelianisation}
        Thompsons's group \(F\) is generated by \(x_0\) and \(x_1\). Moreover,
        the abelianisation of \(F\) is \(\Z^2\), with
        the abelianisation map taking \(x_0\) and \(x_1\) to the
        generators of \(\mathbb{Z}^2\).
    \end{theorem}

    In light of \cref{F:abelianisation}, we can define exponent sums in \(F\).

    \begin{ntn}
        Let \(\varphi \colon F \to \mathbb{Z}^2\) be the abelianisation map. If \(f \in F\),
        then \((f)\varphi = (x_0 \varphi)^i (x_1 \varphi)^j\), for some unique \(i, j \in \mathbb{Z}\).
        Define \(\expsum_{x_0}(f) = i\) and \(\expsum_{x_1}(f) = j\).
    \end{ntn}

    We require a description of the centralisers of the elements of \(F\). We
    use the description given in Matucci's thesis \cite{MatucciThesis}, however
    this is not the only such description.
    The references \cite[Theorem 15.35]{GubaSapir97} and \cite[Theorem 5.5]{BrinSquier} for this fact also describe centralisers in Thompson's group
    \(F\), but require interpretation. 
\begin{proposition}[Proof of {\cite[Proposition 4.2.2]{MatucciThesis}{}}]\label{FCentralisers}
Suppose that \(f\in F\) and let \(0=d_0< d_1<\ldots< d_k=1\) be the unique dyadic rationals fixed by \(f\) such that for all \(i<k\) we have exactly one of:
\begin{itemize}
    \item The interval \((d_i, d_{i+1})\) is a maximal open interval which is fixed pointwise by \(f\);
    \item The interval \((d_{i}, d_{i+1})\) is a maximal open interval such that every dyadic rational within \((d_{i}, d_{i+1})\) lies within the support of \(f\).
\end{itemize}
In this case \(g\in F\) commutes with \(f\) if and only if both of the following hold:
\begin{enumerate}
    \item Each of \(d_0, d_1,\ldots, d_k\) is fixed by \(g\);
    \item For all \(i<k\), \(g|_{(d_i, d_{i+1})}\) commutes with \(f|_{(d_i, d_{i+1})}\).
\end{enumerate}
Moreover, if \(f|_{(d_i, d_{i+1})}\) is non-trivial then the group of all restrictions \(g|_{(d_i, d_{i+1})}\) where \(g\in F\) commutes with \(f\), is a cyclic group.
\end{proposition}

\subsection{Equations in groups}
    We give the basic definitions of equations in groups.
	\begin{dfn}
		Let \(G\) be a finitely generated group, \(V\) be a finite set, and \(F(V)\)
		be the free group on \(V\). An \emph{equation} in \(G\) is an element \(w
		\in G \ast F(V)\) and denoted \(w = 1\). A \emph{solution} to \(w = 1\) is
		a homomorphism \(\phi \colon G \ast F(V) \to G\) which fixes
		\(G\) pointwise, such that \((w) \phi = 1\). The elements of \(V\) are called the
		\emph{variables} of the equation. A \emph{(finite) system of equations} in \(G\)
		is a finite collection of equations in \(G\). A \emph{solution} to a system \(\mathcal E\) of equations
		is a homomorphism that is a solution to every equation in \(\mathcal E\).

        The \emph{Diophantine problem} in \(G\) is the decidability question asking if there is an algorithm that takes as input a system \(\mathcal E\) of equations in \(G\) and outputs whether or not \(\mathcal E\) admits a solution.
	\end{dfn}
    
		We will often abuse notation, and consider a solution to an equation with \(n \in \mathbb{Z}_{>0}\) variables \(X_1, \ldots, X_n\) in a
		group \(G\) to be an \(n\)-tuple of elements \((g_1, \ \ldots, \ g_n)\), rather
		than a homomorphism from \(G \ast F(X_1, \ \ldots, \ X_n) \to G\). A homomorphism
		\(\phi\) can be recovered from such an \(n\)-tuple by setting \(\phi\) to be the homomorphic extension defined using \((g) \phi = g\) if \(g \in G\) and \((X_i)
		\phi = g_i\).

    \begin{ex}
        Consider the equation \(XYX^{-1}Y^{-1}\) in a free group \(F(a, b)\). The set of solutions
        to this will be all pairs \((x, y) \in F(a, b) \times F(a, b)\) such that \(x\) and \(y\) commute. That
        is, all pairs \((x, y) \in F(a, b) \times F(a, b)\) such that there exists \(z \in F(a, b)\) with \(x, y \in \langle z \rangle\).
    \end{ex}

    \begin{dfn}
        Let \(G\) be a finitely generated group and \(n \in \mathbb{Z}_{>0}\). A subset \(S \subseteq G^n\) is called
        \emph{equationally definable} in \(G\) if there is a system
        \(\mathcal E\) of equations in \(G\) with \(m \geq n\)
        variables, such that there are \(n\) variables
        \(X_1, \ldots, X_n\) such that \(S\) is the projection of the set of solutions (when viewed as \(m\)-tuples)
        to \(\mathcal E\), onto the coordinates
        corresponding to \(X_1, \ldots, X_n\). That is, the set of solutions to \(X_1, \ldots, X_n\) within
        the system \(\mathcal E\).
    \end{dfn}

    \begin{ex}
        Let \(G\) be a group. The centraliser \(C_G(g)\) of any element
        \(g \in G\) is equationally definable as it is
        the set of solutions (to \(X\)) in the equation
        \(Xg = gX\).
    \end{ex}

    \begin{dfn}\label{TDef}
        \emph{Thompson's group \(T\)} is the group of piecewise linear homeomorphisms from
        \(S^1 = [0, 1] / (0 = 1)\) to itself that map dyadic rational numbers to dyadic rational
        numbers, differentiable everywhere except at finitely many dyadic rational
        numbers, and such that the derivative on at every differentiable point is always
        an integer power of \(2\).
    \end{dfn}

\subsection{Quadratic equations in the ring of integers}
Our source of undecidability will be Hilbert's tenth problem: the undecidability of the
satisfiability of systems of polynomial equations in non-negative integers. We cover the background and the exact form of Hilbert's tenth problem that we use. 

\begin{dfn}
 A \emph{polynomial equation} in (the ring) \(\Z\) is a pair \((p, q)\) of polynomials over \(\Z\). That is to say, there is a finite set \(V=\{X_1, X_2, \ldots, X_n\}\) (whose elements we call \emph{variables}) such that \(p, q\in \Z[X_1, X_2, \ldots, X_n]\). We usually write \(p=q\) instead of \((p, q)\). 
 A \emph{solution} is a ring homomorphism \(\phi \colon \Z[X_1, X_2, \ldots, X_n]\to \Z\) which fixes \(\Z\) and such that \((p)\phi=(q)\phi\).

A \emph{polynomial equation} in \(\N\) is simply an equation \(p=q\) in \(\Z\) such that all the coefficients of \(p\) and \(q\) are non-negative, and when discussing an equation in \(\N\) we only consider solutions which assign non-negative values to all variables.

A \emph{(finite) system of polynomial equations} in \(\Z\) or \(\N\) is a finite collection of polynomial equations in  \(\Z\) or \(\N\), respectively. 
A \emph{solution} to a system \(\mathcal E\) of polynomial equations in \(\Z\) or \(\N\)
		is a homomorphism that is a solution to every equation in \(\mathcal E\).
\end{dfn}

It is a well-known consequence of Lagrange's four-squares theorem that the decidability of
the satisfiability of systems of polynomial equations in the ring of integers is equivalent
to the decidability of the satisfiability via non-negative integer solutions to systems of polynomial equations in the ring of integers (see, for example, the introduction of \cite{Matijasevic72}, among many others). We thus state the undecidability of Hilbert's tenth problem in the following form:

\begin{theorem}[\cite{Matijasevic70}]
    It is undecidable whether a system of polynomial equations in \(\Z\) admits a solution of
    non-negative integers.
\end{theorem}

An equation \(p=q\) in \(\N\) or \(\Z\) is always equivalent (in the sense of having the same set of solutions) to the equation \(p-q=0\). Thus for our purposes it is sufficient to consider only equations of the form \(p=0\). 
Note that in the case of \(\N\) we might have that \(p-q\) no longer has coefficients in \(\N\). In fact, as every polynomial with integer coefficients is the difference of two polynomials with non-negative coefficients, it follows that solving equations in \(\N\) is equivalent to determining if an equation of the form \(p=0\) (where \(p\) might have some negative coefficients) has solution which assigns variables to non-negative values. From this
we have the following:

\begin{cor}
    It is undecidable whether a system of polynomial equations in \(\N\) admits a solution.
\end{cor}

The final result we mention is also well-known and (usually using \(\Z\) rather than \(\N\)) is
frequently used to show Diophantine problems or similar are undecidable in various classes of
groups and monoids, such as \cite{BuchiSenger,Dong, DuchinLiangShapiro, CiobanuGarreta, Romankov, GarretaMiasnikovOvchinnikov21, GarretaMiasnikovOvchinnikov17}.

\begin{cor}
    \label{cor:dp-naturals-undecidable}
    It is undecidable whether a system of equations in \(\N\) of the form
    \(\gamma_1 + \gamma_2 = \gamma_3\) or \(\gamma_1 \gamma_2
    = \gamma_3\), where each \(\gamma_i\) can be a constant or a variable, admits a solution.
\end{cor}
\begin{proof}
    We show this by showing that for any system \(\mathcal{E}\) of polynomial equations, one can (algorithmically) find a system of equations of the type in this corollary with the same set of solutions.

    For each equation \(p=q\in \mathcal{E}\), we have by definition of a polynomial that there are \(k, l\in \N\) such that
    \[p=\sum_{i<k}\prod_{j<l} p_{i, j},\quad \quad q=\sum_{i<k}\prod_{j<l} q_{i, j},\]
    where \(p_{i, j}\) and \(q_{i, j}\) are either variables or constants in \(\N\). 
    Thus we can represent this by thinking of \(p\) and \(q\) as variable symbols and add equations insisting that these are built up in the required way. 
\end{proof}

\section{Proof of main theorem}

We begin the proof of \cref{main-thm}. Our first task will be to show that various
subgroups and submonoids of Thompson's group \(F\) are equationally definable. To avoid any confusion between subgroups and submonoids, we introduce the following notation:
\begin{ntn}
    If \(G\) is a group and \(g \in G\), we use \(\langle g \rangle\) to denote the subgroup
    generated by \(g\) and \(\Mon\langle g \rangle\) to denote the submonoid generated by \(g\).
\end{ntn}

\begin{lem}\label{monogenicmonoiddefinable}
  The following sets are equationally definable in Thompson's
  group \(F\):
  \begin{enumerate}
      \item \(\langle x_0 \rangle\);
      \item \(\langle x_1 \rangle\);
      \item \(\Mon\langle x_0 \rangle\);
      \item \(\Mon\langle x_1 \rangle\);
      \item \(\Mon\langle x_0 x_1 \rangle\);
      \item \(\{(r, s) \mid r \in \Mon\langle x_1 \rangle, s \in \Mon\langle x_0r \rangle\}\).
  \end{enumerate}
\end{lem}

\begin{proof}
 It  immediately follows from \cref{FCentralisers} that \(C_F(x_0)\) is a cyclic subgroup of
\(F\) containing \(x_0\), and is thus \(\langle x_0 \rangle\) (because the derivative of \(x_0\) at \(0\) is \(2^1\) so \(x_0\) is not a proper power of any element of \(F\)). 
It follows that
\(\langle x_0 \rangle\) is equationally definable, as the set of solutions to \(Xx_0  =x_0 X\),
and we have shown (1).

To show (2), let \(x_0', x_1'\) generate the group of elements of \(F\) supported on \((0,\frac{1}{2})\) (which is isomorphic to \(F\) by \cite[Lemma 4.4]{CannonFloydParry}).
Since \(F\) has trivial centre (corollary to \cref{FCentralisers} or \cite[Theorem 4.3 and Theorem 4.5]{CannonFloydParry}), the group of elements of \(F\) which commute with \(x_0'\) and \(x_1'\) is just the set of elements
of \(F\) whose support does not intersect \((0, \frac{1}{2})\). 
Thus by \cref{FCentralisers}, the elements commuting with \(x_1\) and \(\langle x_0', x_1' \rangle\) are precisely
those whose support lies in \((\frac{1}{2}, 1)\) and commute with \(x_1\). By \cref{FCentralisers}, this is a cyclic group containing \(x_1\), and is
thus \(\langle x_1 \rangle\) (as the derivative of \(x_1\) at \(1\) is \(2^{-1}\) and is this not a proper power of any element of \(F\)). We can conclude that \(\langle x_1 \rangle\)
is equationally definable, and we have thus shown (2).

By the chain rule, for all \(k\in \N\) the element \(x_0 x_1^k\) has derivative \(2\cdot 1^k=2\) at \(0\). In addition, \(x_0x_1^k\) has support \((0, 1)\) because for all \(t \in (0, 1)\) we have \((t)x_0 > t\) and \((t)x_1 \geq t\) and hence \((t)x_0 x_1^k > t\). \cref{FCentralisers} thus tells us that the centraliser of \(x_0 x_1^k\) is cyclic. As \(x_0 x_1^k\) has derivative \(2^1\) at \(0\), it is not a proper power, and must therefore generate its centraliser. 

Let \(l\in F\) be any element with support \((0, \frac{1}{2})\) (it follows from the definition of \(F\) that such an element exists). If \(h = (x_0 x_1^k)^n\) for some \(k\in \N\), \(n\in \Z\), then \
\begin{align*}
    n\geq 0 &\iff \left(\frac{1}{2}\right) h \geq \frac{1}{2}\\
    & \iff \left(\frac{1}{2} \right) h \not\in \supt(l)\\
    & \iff \left(\frac{1}{2} , 1\right) h \cap \supt(l) = \varnothing \\
    & \iff \supt(h^{-1}x_1 h)\cap \supt(l) = \varnothing \\
    &\iff[h^{-1} x_1 h, l]=1.
\end{align*}
Thus for all \(k\in \N \), we have 
    \begin{equation}
    \label{abk-definable}   
    \Mon\langle x_0 x_1^k \rangle=\makeset{h\in F}{\(h x_0 x_1^k = x_0 x_1^k h\) and \([h^{-1} x_1 h, l]=1\)},
    \end{equation}
    and thus \(\Mon \langle x_0 x_1^k \rangle\) is equationally definable. In
    particular, we have shown (5). The case of \(k = 0\)
    gives that \(\Mon\langle x_0 \rangle\) is equationally definable, and thus we have shown (3).

To show that \(\Mon \langle x_1 \rangle\) is equationally definable, as with (2), let \(x_0', x_1'\) generate the group of elements of \(F\) supported on \((0,\frac{1}{2})\) (which is isomorphic to \(F\) by \cite[Lemma 4.4]{CannonFloydParry}).
The group of elements of \(F\) which commute with \(x_0'\) and \(x_1'\) is canonically isomorphic to \(F\) via an isomorphism \(\phi\) taking \(x_0\) to \(x_1\). 
    Then using \eqref{abk-definable} with \(k = 0\),
     \[\Mon \langle x_1 \rangle=\makeset{h\in F}{\(hx_1 = x_1 h\) and \(h x_1' = x_1'h\) and \(h x_0' = x_0'x\) and \([h^{-1}(x_1)\phi h, (l)\phi]=1\)},\]
     and so \(\Mon\langle x_1 \rangle\) is equationally definable, and we have shown (4).
    
    To show (6), if \(\mathcal E\) is a system of equations with
    a variable \(X\), such that the set of solutions to \(X\) is \(\Mon\langle x_1 \rangle\), then 
    using \eqref{abk-definable}, the set of solutions to \((X, Y)\) in \(\mathcal E \cup \{Y (x_0 X) = (x_0 X)Y, [Y^{-1}x_1Y, l] = 1\}\) will be  \(\{(r, s) \mid r \in \Mon\langle x_1 \rangle, s \in \Mon\langle x_0r \rangle\}\), and so we have shown (6).
\end{proof}

    Our simulation of abelianisation constraints in \(F\) is captured by the following result. This is an immediate corollary of the fact that \(F\) has
    finite commutator width \cite{commwidth}.

    \begin{lem}
        \label{lem:comm-defable}
        The commutator subgroup \([F, F]\) of Thompson's group \(F\) is equationally
        definable in \(F\).
    \end{lem}

    \begin{proof}
        By \cite[Theorem 1.1]{commwidth}, \([F, F] = \{[g_1, g_2][g_3, g_4] \mid
        g_1, g_2, g_3, g_4 \in F\}\), which is just the set of solutions to \(X\)
        in the equation \(X = [Y_1, Y_2][Y_3, Y_4]\).
    \end{proof}
    
    The exact proof of the undecidability of the Diophantine problem with
    abelianisation constraints in free groups in \cite{CiobanuGarreta} first
    shows the undecidability of the Diophantine problem with exponent-sum
    constraints. We capture these constraints in the following lemma:

    \begin{lem}
        \label{lem:expsums-defable}
        The sets
        \begin{enumerate}
            \item \(\{g \in F \mid \expsum_{x_0}(g) = 0\}\);
            \item \(\{g \in F \mid \expsum_{x_1}(g) = 0\}\);
            \item \(\{(g,h) \in F\times F \mid \expsum_{x_0}(g) = \expsum_{x_0}(h)\}\);
            \item \(\{(g,h) \in F\times F \mid \expsum_{x_1}(g) = \expsum_{x_1}(h)\}\);
            \item \(\{(g,h) \in F\times F \mid \expsum_{x_0}(g) = \expsum_{x_1}(h)
            \geq 0\}\).
        \end{enumerate}
        are all equationally definable in \(F\).
    \end{lem}

    \begin{proof}
        We show \(\{g \in F \mid \expsum_{x_0}(g) = 0\} = [F, F] \langle x_1 \rangle\)
        and \(\{g \in F \mid \expsum_{x_1}(g) = 0\} = [F, F] \langle x_0 \rangle\). (1) and (2) will then follow from the fact that products of equationally definable
        sets are equationally definable, together with the facts that
        \([F, F]\), \(\langle x_0 \rangle\) and \(\langle x_1 \rangle\) are
        all equationally definable (\cref{monogenicmonoiddefinable} and \cref{lem:comm-defable}). Clearly,
        \(\{g \in F \mid \expsum_{x_0}(g) = 0\} \supseteq [F, F] \langle x_1 \rangle\). So
        let \(g \in F\) be such that \(\expsum_{x_0}(g) = 0\). Then \(\expsum_{x_0}(g x_1^{-\expsum_{x_1}(g)}) = 0 = \expsum_{x_1}(g x_1^{-\expsum_{x_1}(g)})\), and so \(g x_1^{-\expsum_{x_1}(g)}
        \in [F, F]\). Thus \(g \in [F, F] \langle x_1 \rangle\). The fact that 
        \(\{g \in F \mid \expsum_{x_1}(g) = 0\} = [F, F] \langle x_0 \rangle\) follows 
        similarly.

        We now show (3) and (4) using the following observations:
\[\{(g,h) \in F \times F \mid \expsum_{x_0}(g) = \expsum_{x_0}(h)\}=\{(g,h) \in F \times F \mid \expsum_{x_0}(gh^{-1}) = 0\},\]
\[\{(g,h) \in F \times F \mid \expsum_{x_1}(g) = \expsum_{x_1}(h)\}=\{(g,h) \in F \times F \mid \expsum_{x_1}(gh^{-1}) = 0\}.\]
Finally, to show (5), the set
\(\{(g,h) \in F \times F \mid \expsum_{x_0}(g) = \expsum_{x_1}(h)\geq 0\}\) is equal to
\[\{(g,h) \in F \times F \mid \exists x\in \Mon\langle x_0 x_1\rangle, \expsum_{x_0}(g) = \expsum_{x_0}(x) \textrm{ and } \expsum_{x_1}(h) = \expsum_{x_1}(x)\}.\]
\cref{monogenicmonoiddefinable} (3) and (4) thus now show that this set is equationally definable.
\end{proof}

    We have now completed all of the setup required to apply certain ideas from \cite[Theorem 3.3]{CiobanuGarreta} to show that the Diophantine problem in
    \(F\) is undecidable.

    \begin{theorem_no_number}[\ref{main-thm}]
        The Diophantine problem in Thompson's group \(F\) is undecidable.
    \end{theorem_no_number}

    \begin{proof}
        We will embed the natural numbers into \(F\) as \(\Mon \langle x_0 \rangle\). 
        First note \(x_0^k x_0^l = x_0^{k + l}\),
        so the system of equations \(XY = Z\), together with the equations that insist that \(X, Y, Z \in \Mon\langle x_0 \rangle\) (which exist by \cref{monogenicmonoiddefinable}), allow us to say that if \(X =x_0^k\) and \(Y=  x_0^l\) then \(Z = x_0^{k + l}\). We have thus shown that \(\{(x_0^k, x_0^l, x_0^{k + l} )\mid k, l \in \N\}\) is equationally definable in \(F\).
        
        We now proceed with multiplication; that is, we show that the set \(\{(x_0^k, x_0^l, x_0^{kl} )\mid k, l \in \N\}\) is equationally
        definable in \(F\). Let \(S \subseteq F \times F\) be the set defined in
        \cref{monogenicmonoiddefinable}(5). 
        Let \(k, l, m \in \N\). We claim that \(kl = m\) if and only if there exist \(r, s \in F\)
        satisfying:
        \begin{align}
            \label{eq:otimes}
            (r, s) \in S \wedge \expsum_{x_0}(x_0^k) = \expsum_{x_1}(r) \wedge
            \expsum_{x_0}(x_0^l) = \expsum_{x_0}(s) \wedge \expsum_{x_0}(x_0^m) = \expsum_{x_1}(s).
        \end{align}
        Suppose there exist \(r, s \in F\) satisfying \eqref{eq:otimes}.
        By the definition of \(S\), \(r = x_1^p\) for some \(p \in \N\) and \(s = (x_0 r)^q\), for some \(q \in \N\),
        and so \(s = (x_0 x_1^p)^q\). Since \(\expsum_{x_0}(x_0^k) = \expsum_{x_1}(r)\) we have
        \(p = k\) and since \(\expsum_{x_0}(x_0^l) = \expsum_{x_0}(s)\), we have \(q = l\). Finally,
        \(\expsum_{x_0}(x_0^m) =\expsum_{x_1}(s)\) tells us that \(m = pq = kl\), and we have
        shown the backwards direction of the claim.

        For the converse, suppose \(m = kl\). Set \(r = x_1^k\) and \(s = (x_0 x_1^k)^l = (x_0 r)^l\).
        Then \((r, s) \in S\). Additionally, \(\expsum_{x_0}(x_0^k) = k = \expsum_{x_1}(r)\), \(\expsum_{x_0}(x_0^l) = l = \expsum_{x_0}(s)\) and
        \(\expsum_{x_0}(x_0^m) = m = kl = \expsum_{x_1}(s)\). Thus \(r\) and \(s\) satisfy the conditions in \eqref{eq:otimes}, as
        required.

        Since the set \(S\) is equationally definable and the
        exponent-sum sets stated in \cref{lem:expsums-defable} are
        equationally definable, it thus
        follows that \(\{(x_0^k, x_0^l, x_0^{kl}) \mid k, l \in \N\}\)
        is equationally definable in
        \(F\). Taking this together with the fact that \(\{(x_0^k, x_0^l, x_0^{k + l}) \mid k, l \in \N\}\) is equationally definable, we can define the solution set to any system of
        polynomial equations over the natural numbers. Since
        the satisfiability of such a system is undecidable \cite{Matijasevic72}, it follows that the Diophantine
        problem of \(F\) is undecidable.
    \end{proof}

We now consider Thompson's group \(T\). To show that \(T\) has an undecidable Diophantine
problem, it suffices to show that a subgroup that is isomorphic to \(F\) is equationally definable in \(T\). In such a case, an algorithm solving the Diophantine problem for \(T\)
could be used to solve the Diophantine problem of \(F\). The fact that \(F\) can be
realised as a centraliser (and thus is equationally definable) in \(T\) was pointed out
to us by James Belk and Corentin Bodart independently. We include the short proof below. General centralisers of \(T\) have been described in \cite{MatucciThesis, RobertsonThesis, GeogheganVarisco}. 

 \begin{cor_no_number}[\ref{main-cor}]
       The Diophantine problem in Thompson's group \(T\) is undecidable.
    \end{cor_no_number}
\begin{proof}
By definition, \(T\) acts on a quotient of \([0, 1]\). When meaning is unambiguous, we write as if \(T\) acts on \([0, 1]\) itself.
    Note that the set \(\makeset{f\in T}{\((0)f=0\)}\) is precisely \(F\) (if we make no distinction between \([0, 1]\) and \([0, 1] / (0 = 1)\)).

    Thus  \(\makeset{f\in T}{\(f\) fixes \([0, \frac{1}{2}]\) pointwise}\) is isomorphic to \(F\), by \cite[Lemma 4.4]{CannonFloydParry}. In particular, if \(x_0'\) and \(x_1'\) are the elements from \cref{monogenicmonoiddefinable}, then the centraliser of \(x_0'\) and \(x_1'\) in \(F\) is the copy \(\makeset{f\in T}{\(f\) fixes \([0, \frac{1}{2}]\) pointwise}\) of \(F\). Every element of \(T\) which commutes with \(x_0'\) must preserve the set \([0, \frac{1}{2}]\) setwise and thus must fix \(0\) and \(\frac{1}{2}\) (any element of \(T\) switching \(0\) and \(\frac{1}{2}\) fails to preserve the set \([0, \frac{1}{2}]\)). It thus follows that the centraliser of \(x_0'\) and
    \(x_1'\) in \(T\) is equal to the centraliser in \(F\) of the same elements, which is
    isomorphic to \(F\).
    As the centraliser of a finite set, it is equationally definable in \(T\), and so
    we have shown that an isomorphic copy of \(F\) is equationally definable in \(T\), and thus \cref{main-thm} tells
    us that \(T\) has an undecidable Diophantine problem.
\end{proof}
\section{Further questions}
Given the various groups that exhibit similar behaviour to Thompson's group \(F\) and \(T\), it is natural to ask whether the Diophantine problem is decidable for them, and whether similar methods can be used.
Perhaps the first example to consider is Thompson's group \(V\). As discussed in the introduction, we do not expect that the arguments used here can be adapted for \(V\) directly, since its abelianisation is trivial. It is true that \(F\), \(T\) and many other groups with undecidable Diophantine problem, such as \(\Z \wr \Z\) \cite{Dong}, embed into \(V\), and so it would be sufficient to show that one of these is equationally definable in \(V\) to show
that the Diophantine problem is undecidable.

\begin{question}
    Is the Diophantine problem decidable in Thompson's group \(V\)?
\end{question}

Other natural directions to consider are towards generalisations of \(F\), such as to diagram groups. Since \(\Z\) is a diagram
group \cite[Example 6.3]{GubaSapir97}, some diagram groups can have decidable Diophantine problem, so one would instead ask which diagram groups have decidable Diophantine problem. Given the description of centralisers in diagram groups \cite{GubaSapir97}, there may be some hope of proving that other diagram groups have undecidable Diophantine problem. To apply a similar argument to that used for Thompson's group \(F\), one would need these diagram groups to have finite commutator width (or at least some method of showing that the commutator subgroup is equationally definable) and abelianisations of rank at least \(2\). 

\begin{question}
    Which diagram groups have decidable Diophantine problem?
\end{question}

Another direction one can consider, is to remain in \(F\) and ask where the boundary of decidability and undecidability lies. Given \(F\) has a decidable conjugacy problem \cite{GubaSapir97} and undecidable Diophantine problem, one natural question that lies between these problems is the question as to whether the satisfiability of single equations, rather than systems, is decidable. We call this decidability question the \emph{single equation problem}. The decidability of the single
equation problem can differ from that of the Diophantine problem, with the Heisenberg group being an
example \cite{DuchinLiangShapiro}. 

\begin{question}
    Is the single equation problem decidable in Thompson's group \(F\)?
\end{question}

One can restrict further still, and look at \emph{quadratic equations}:
equations where each variable occurs exactly twice. This area has seen a lot of
interest recently with a wide variety of positive decidability results,
including the Grigorchuk group \cite{LysenokMiasnikovUshakov}, solvable
Baumslag-Solitar groups \cite{LysenokUshakov}, 
and Lamplighter groups \cite{KharlampovichLopezMyasnikov, UshakovWeiers}, amongst many others.

\begin{question}
    Is it decidable whether a quadratic equation in \(F\) admits a solution?
\end{question}

\section*{Acknowledgements}
The authors would like to thank James Belk and Corentin Bodart for pointing out that \(F\) is
isomorphic to a centraliser in \(T\), and Collin Bleak, Corentin Bodart and Matthew Brin for finding and explaining references. The authors would like to thank the anonymous referee
for their helpful comments.
Elliott was supported by the Heilbronn Institute for Mathematical Research. Levine was
supported by the EPSRC Fellowship grant EP/V032003/1 ‘Algorithmic, topological and geometric aspects of infinite groups, monoids and inverse semigroups’.

\bibliography{references}

\begin{bibdiv}
\begin{biblist}

\bib{BardakovTolstykh}{article}{
      author={Bardakov, Valery},
      author={Tolstykh, Vladimir},
       title={Interpreting the arithmetic in {T}hompson's group {$F$}},
        date={2007},
        ISSN={0022-4049,1873-1376},
     journal={J. Pure Appl. Algebra},
      volume={211},
      number={3},
       pages={633\ndash 637},
         url={https://doi.org/10.1016/j.jpaa.2007.02.011},
      review={\MR{2344220}},
}

\bib{BelkMatucci}{article}{
      author={Belk, James},
      author={Matucci, Francesco},
       title={Conjugacy and dynamics in {T}hompson's groups},
        date={2014},
        ISSN={0046-5755,1572-9168},
     journal={Geom. Dedicata},
      volume={169},
       pages={239\ndash 261},
         url={https://doi.org/10.1007/s10711-013-9853-2},
      review={\MR{3175247}},
}

\bib{Brin05}{article}{
      author={Brin, Matthew~G.},
       title={Coherence of associativity in categories with multiplication},
        date={2005},
        ISSN={0022-4049,1873-1376},
     journal={J. Pure Appl. Algebra},
      volume={198},
      number={1-3},
       pages={57\ndash 65},
         url={https://doi.org/10.1016/j.jpaa.2004.10.008},
      review={\MR{2132873}},
}

\bib{BrinSquier}{article}{
      author={Brin, Matthew~G.},
      author={Squier, Craig~C.},
       title={Presentations, conjugacy, roots, and centralizers in groups of piecewise linear homeomorphisms of the real line},
        date={2001},
        ISSN={0092-7872,1532-4125},
     journal={Comm. Algebra},
      volume={29},
      number={10},
       pages={4557\ndash 4596},
         url={https://doi.org/10.1081/AGB-100106774},
      review={\MR{1855112}},
}

\bib{BuchiSenger}{article}{
      author={B\"uchi, J.~Richard},
      author={Senger, Steven},
       title={Definability in the existential theory of concatenation and undecidable extensions of this theory},
        date={1988},
        ISSN={0044-3050},
     journal={Z. Math. Logik Grundlag. Math.},
      volume={34},
      number={4},
       pages={337\ndash 342},
         url={https://doi.org/10.1002/malq.19880340410},
      review={\MR{950368}},
}

\bib{CannonFloydParry}{article}{
      author={Cannon, J.~W.},
      author={Floyd, W.~J.},
      author={Parry, W.~R.},
       title={Introductory notes on {R}ichard {T}hompson's groups},
        date={1996},
        ISSN={0013-8584},
     journal={Enseign. Math. (2)},
      volume={42},
      number={3-4},
       pages={215\ndash 256},
      review={\MR{1426438}},
}

\bib{CiobanuEvettsLevine}{article}{
      author={Ciobanu, Laura},
      author={Evetts, Alex},
      author={Levine, Alex},
       title={Effective {E}quation {S}olving, {C}onstraints, and {G}rowth in {V}irtually {A}belian {G}roups},
        date={2025},
        ISSN={2470-6566},
     journal={SIAM J. Appl. Algebra Geom.},
      volume={9},
      number={1},
       pages={235\ndash 260},
         url={https://doi.org/10.1137/23M1604679},
      review={\MR{4876546}},
}

\bib{CiobanuGarreta}{article}{
      author={Ciobanu, Laura},
      author={Garreta, Albert},
       title={Group equations with abelian predicates},
        date={2024},
        ISSN={1073-7928,1687-0247},
     journal={Int. Math. Res. Not. IMRN},
      number={5},
       pages={4119\ndash 4159},
         url={https://doi.org/10.1093/imrn/rnad179},
      review={\MR{4714338}},
}

\bib{CiobanuHoltRees}{article}{
      author={Ciobanu, Laura},
      author={Holt, Derek},
      author={Rees, Sarah},
       title={Equations in groups that are virtually direct products},
        date={2020},
        ISSN={0021-8693},
     journal={J. Algebra},
      volume={545},
       pages={88\ndash 99},
         url={https://doi.org/10.1016/j.jalgebra.2018.10.044},
      review={\MR{4044690}},
}

\bib{DahmaniGuirardel}{article}{
      author={Dahmani, Fran\c{c}ois},
      author={Guirardel, Vincent},
       title={Foliations for solving equations in groups: free, virtually free, and hyperbolic groups},
        date={2010},
        ISSN={1753-8416},
     journal={J. Topol.},
      volume={3},
      number={2},
       pages={343\ndash 404},
         url={https://doi.org/10.1112/jtopol/jtq010},
      review={\MR{2651364}},
}

\bib{Dehornoy96}{article}{
      author={Dehornoy, Patrick},
       title={The structure group for the associativity identity},
        date={1996},
        ISSN={0022-4049,1873-1376},
     journal={J. Pure Appl. Algebra},
      volume={111},
      number={1-3},
       pages={59\ndash 82},
         url={https://doi.org/10.1016/0022-4049(95)00119-0},
      review={\MR{1394345}},
}

\bib{Dehornoy2005}{article}{
      author={Dehornoy, Patrick},
       title={Geometric presentations for {T}hompson's groups},
        date={2005},
        ISSN={0022-4049,1873-1376},
     journal={J. Pure Appl. Algebra},
      volume={203},
      number={1-3},
       pages={1\ndash 44},
         url={https://doi.org/10.1016/j.jpaa.2005.02.012},
      review={\MR{2176650}},
}

\bib{DiekertLohrey}{incollection}{
      author={Diekert, Volker},
      author={Lohrey, Markus},
       title={Existential and positive theories of equations in graph products},
        date={2002},
   booktitle={S{TACS} 2002},
      series={Lecture Notes in Comput. Sci.},
      volume={2285},
   publisher={Springer, Berlin},
       pages={501\ndash 512},
         url={https://doi.org/10.1007/3-540-45841-7_41},
      review={\MR{2050863}},
}

\bib{DiekertMuscholl}{incollection}{
      author={Diekert, Volker},
      author={Muscholl, Anca},
       title={Solvability of equations in free partially commutative groups is decidable},
        date={2001},
   booktitle={Automata, languages and programming},
      series={Lecture Notes in Comput. Sci.},
      volume={2076},
   publisher={Springer, Berlin},
       pages={543\ndash 554},
         url={https://doi.org/10.1007/3-540-48224-5_45},
      review={\MR{2066532}},
}

\bib{Dong}{inproceedings}{
      author={Dong, Ruiwen},
       title={Linear equations with monomial constraints and decision problems in abelian-by-cyclic groups},
        date={2025},
   booktitle={Proceedings of the 2025 {A}nnual {ACM}-{SIAM} {S}ymposium on {D}iscrete {A}lgorithms ({SODA})},
   publisher={SIAM, Philadelphia, PA},
       pages={1892\ndash 1908},
         url={https://doi.org/10.1137/1.9781611978322.59},
      review={\MR{4863474}},
}

\bib{DuchinLiangShapiro}{article}{
      author={Duchin, Moon},
      author={Liang, Hao},
      author={Shapiro, Michael},
       title={Equations in nilpotent groups},
        date={2015},
        ISSN={0002-9939,1088-6826},
     journal={Proc. Amer. Math. Soc.},
      volume={143},
      number={11},
       pages={4723\ndash 4731},
         url={https://doi.org/10.1090/proc/12630},
      review={\MR{3391031}},
}

\bib{commwidth}{article}{
      author={Gal, \'Swiatos\l aw~R.},
      author={Gismatullin, Jakub},
       title={Uniform simplicity of groups with proximal action},
        date={2017},
        ISSN={2330-0000},
     journal={Trans. Amer. Math. Soc. Ser. B},
      volume={4},
       pages={110\ndash 130},
         url={https://doi.org/10.1090/btran/18},
        note={With an appendix by Nir Lazarovich},
      review={\MR{3693109}},
}

\bib{GarretaMiasnikovOvchinnikov21}{article}{
      author={Garreta, Albert},
      author={Legarreta, Leire},
      author={Miasnikov, Alexei},
      author={Ovchinnikov, Denis},
       title={Metabelian groups: full-rank presentations, randomness and diophantine problems},
        date={2021},
        ISSN={1433-5883,1435-4446},
     journal={J. Group Theory},
      volume={24},
      number={3},
       pages={453\ndash 466},
         url={https://doi.org/10.1515/jgth-2020-0091},
      review={\MR{4250503}},
}

\bib{GarretaMiasnikovOvchinnikov17}{article}{
      author={Garreta, Albert},
      author={Miasnikov, Alexei},
      author={Ovchinnikov, Denis},
       title={Random nilpotent groups, polycyclic presentations, and {D}iophantine problems},
        date={2017},
        ISSN={1867-1144,1869-6104},
     journal={Groups Complex. Cryptol.},
      volume={9},
      number={2},
       pages={99\ndash 115},
         url={https://doi.org/10.1515/gcc-2017-0007},
      review={\MR{3717096}},
}

\bib{GeogheganVarisco}{incollection}{
      author={Geoghegan, Ross},
      author={Varisco, Marco},
       title={On {T}hompson's group {$T$} and algebraic {$K$}-theory},
        date={2018},
   booktitle={Geometric and cohomological group theory},
      series={London Math. Soc. Lecture Note Ser.},
      volume={444},
   publisher={Cambridge Univ. Press, Cambridge},
       pages={34\ndash 45},
      review={\MR{3822287}},
}

\bib{GubaSapir97}{article}{
      author={Guba, Victor},
      author={Sapir, Mark},
       title={Diagram groups},
        date={1997},
        ISSN={0065-9266,1947-6221},
     journal={Mem. Amer. Math. Soc.},
      volume={130},
      number={620},
       pages={viii+117},
         url={https://doi.org/10.1090/memo/0620},
      review={\MR{1396957}},
}

\bib{KharlampovichLopezMyasnikov}{article}{
      author={Kharlampovich, Olga},
      author={Lopez, Laura},
      author={Miasnikov, Alexei},
       title={Diophantine problem in some metabelian groups},
        date={2023},
     journal={ar{X}iv e-prints},
}

\bib{Lasserre}{article}{
      author={Lasserre, Cl\'ement},
       title={R. {J}. {T}hompson's groups {$F$} and {$T$} are bi-interpretable with the ring of the integers},
        date={2014},
        ISSN={0022-4812,1943-5886},
     journal={J. Symb. Log.},
      volume={79},
      number={3},
       pages={693\ndash 711},
         url={https://doi.org/10.1017/jsl.2014.29},
      review={\MR{3248780}},
}

\bib{LysenokMiasnikovUshakov}{article}{
      author={Lysenok, Igor},
      author={Miasnikov, Alexei},
      author={Ushakov, Alexander},
       title={Quadratic equations in the {G}rigorchuk group},
        date={2016},
        ISSN={1661-7207,1661-7215},
     journal={Groups Geom. Dyn.},
      volume={10},
      number={1},
       pages={201\ndash 239},
         url={https://doi.org/10.4171/GGD/348},
      review={\MR{3460336}},
}

\bib{LysenokUshakov}{article}{
      author={Lysenok, Igor},
      author={Ushakov, Alexander},
       title={Orientable quadratic equations in free metabelian groups},
        date={2021},
        ISSN={0021-8693,1090-266X},
     journal={J. Algebra},
      volume={581},
       pages={303\ndash 326},
         url={https://doi.org/10.1016/j.jalgebra.2021.04.013},
      review={\MR{4256899}},
}

\bib{Makanin_systems}{article}{
      author={Makanin, G.~S.},
       title={Systems of equations in free groups},
        date={1972},
        ISSN={0037-4474},
     journal={Sibirsk. Mat. \v{Z}.},
      volume={13},
       pages={587\ndash 595},
      review={\MR{0318314}},
}

\bib{Makanin_semigroups}{article}{
      author={Makanin, G.~S.},
       title={The problem of the solvability of equations in a free semigroup},
        date={1977},
     journal={Mat. Sb. (N.S.)},
      volume={103(145)},
      number={2},
       pages={147\ndash 236, 319},
      review={\MR{0470107}},
}

\bib{Makanin_eqns_free_group}{article}{
      author={Makanin, G.~S.},
       title={Equations in a free group},
        date={1982},
        ISSN={0373-2436},
     journal={Izv. Akad. Nauk SSSR Ser. Mat.},
      volume={46},
      number={6},
       pages={1199\ndash 1273, 1344},
      review={\MR{682490}},
}

\bib{Matijasevic70}{article}{
      author={Matijasevi\v~c, Ju.\~V.},
       title={The {D}iophantineness of enumerable sets},
        date={1970},
        ISSN={0002-3264},
     journal={Dokl. Akad. Nauk SSSR},
      volume={191},
       pages={279\ndash 282},
      review={\MR{258744}},
}

\bib{Matijasevic72}{article}{
      author={Matijasevi\v~c, Ju.\~V.},
       title={Diophantine sets},
        date={1972},
        ISSN={0042-1316},
     journal={Uspehi Mat. Nauk},
      volume={27},
      number={5(167)},
       pages={185\ndash 222},
      review={\MR{441711}},
}

\bib{MatucciThesis}{thesis}{
      author={Matucci, Francesco},
       title={Algorithms and classification in groups of piecewise-linear homeomorphisms},
        type={Ph.D. Thesis},
        date={2008},
}

\bib{McKenzieThompson}{incollection}{
      author={McKenzie, Ralph},
      author={Thompson, Richard~J.},
       title={An elementary construction of unsolvable word problems in group theory},
        date={1973},
   booktitle={Word problems: decision problems and the {B}urnside problem in group theory ({C}onf., {U}niv. {C}alifornia, {I}rvine, {C}alif., 1969; dedicated to {H}anna {N}eumann)},
      series={Stud. Logic Found. Math.},
      volume={Vol. 71},
   publisher={North-Holland, Amsterdam-London},
       pages={457\ndash 478},
         url={https://doi.org/10.1016/0003-4916(72)90140-6},
      review={\MR{396769}},
}

\bib{RipsSela}{article}{
      author={Rips, E.},
      author={Sela, Z.},
       title={Canonical representatives and equations in hyperbolic groups},
        date={1995},
        ISSN={0020-9910},
     journal={Invent. Math.},
      volume={120},
      number={3},
       pages={489\ndash 512},
         url={https://doi.org/10.1007/BF01241140},
      review={\MR{1334482}},
}

\bib{RobertsonThesis}{thesis}{
      author={Robertson, David~Matthew},
       title={Conjugacy and centralisers in thompson’s group t},
        type={Ph.D. Thesis},
        date={2019},
}

\bib{Romankov79}{article}{
      author={Roman\cprime~kov, V.~A.},
       title={Equations in free metabelian groups},
        date={1979},
        ISSN={0037-4474},
     journal={Sibirsk. Mat. Zh.},
      volume={20},
      number={3},
       pages={671\ndash 673, 694},
      review={\MR{537377}},
}

\bib{Romankov}{article}{
      author={Roman\cprime~kov, V.~A.},
       title={Universal theory of nilpotent groups},
        date={1979},
        ISSN={0025-567X},
     journal={Mat. Zametki},
      volume={25},
      number={4},
       pages={487\ndash 495, 635},
      review={\MR{534291}},
}

\bib{Thompson}{misc}{
      author={Thompson, Richard},
       title={Unpublished but widely circulated handwritten notes},
        date={1965},
}

\bib{UshakovWeiers}{article}{
      author={Ushakov, Alexander},
      author={Weiers, Chloe},
       title={Quadratic equations in the lamplighter group},
        date={2025},
        ISSN={0747-7171,1095-855X},
     journal={J. Symbolic Comput.},
      volume={129},
       pages={Paper No. 102417, 18},
         url={https://doi.org/10.1016/j.jsc.2024.102417},
      review={\MR{4840449}},
}

\end{biblist}
\end{bibdiv}
\bibliographystyle{alpha}
\end{document}